\documentclass{article}
\usepackage[utf8]{inputenc}

\usepackage[usenames]{color}
\usepackage[T1]{fontenc}
\usepackage{amsmath}
\usepackage{amsthm}
\usepackage{amssymb}
\usepackage{amsfonts}

\usepackage{graphicx}
\usepackage{mathrsfs}
\usepackage{textcomp}
\usepackage{bm}
\usepackage{blindtext}
\usepackage{imakeidx}

\usepackage{lipsum}

\newcommand\summaryname{Abstract}
\newenvironment{Abstract}%
    {\small\begin{center}%
    \bfseries{\summaryname} \end{center}}

\usepackage{caption}
\definecolor{MyDarkGreen}{rgb}{0,0.2,0}
\definecolor{MyGreen}{rgb}{0,0.6,0}
\DeclareMathOperator*{\divergenza}{div}

\newtheorem{theorem}{Theorem}[section]
\newtheorem{cor}{Corollary}[theorem]
\newtheorem{prop}[theorem]{Proposition}
\newtheorem{lemma}[theorem]{Lemma}
\newtheorem{definition}[theorem]{Definition}

\usepackage{enumitem}
\newlist{steps}{enumerate}{1}
\setlist[steps, 1]{label = Step \arabic*:}

\usepackage{sectsty} 

\title{A stationary solution for the mixed turbulence shell model}
\author{Alessandro Montagnani}

\allowdisplaybreaks

\begin{document}

\maketitle

\begin{Abstract}
The aim of this work is to prove an existence result on the mixed shell model extending the classic standard existence results from $\ell^2$ initial conditions to $\mu$-almost every initial conditions, where $\mu$ is a Gaussian measure on the infinite dimensional space of initial conditions.





A similar result is also shown to hold for turbulence models where a dyadic tree structure replaces the linear one.

\end{Abstract}

\begin{section}{Introduction}
The dynamic of a fluid can be described by the Navier-Stokes equation. However, there are still aspects of fluid dynamics that are not completely understood, such as turbulence. For this reason, several simplified models have been introduced, including shell models. They have now a long and established history, starting with Novikov in the 1970s~\cite{DesNov1974}. Since then, they have been extensively studied, in several different variants, see for example~\cite{Bif2003} or~\cite{AleBif2018}, and many results are now known, see~\cite{7} for a deeper dive.
From a mathematical perspective, different model choices are possible, focusing on different properties of the fluids that one wants to investigate with these simplified models. Among such models, we can consider the dyadic models and their variants, with nonlinearity as the one introduced by Novikov or the one discussed by Obukhov: a discussion concerning these modes can be found in~\cite{Wal2006}.
To give here just a quick idea, these models focus on the energy cascade, that is the transfer of (kinetic) energy in the fluid between different scales.
The nonlinearity chosen by Novikov~\cite{DesNov1974} is particularly well suited to model the forward cascade, in which the energy moves from large scales to small ones, whereas the Obukhov~\cite{Obu1971} nonlinearity captures the backward cascade, with energy crawling up from small scales to big ones. 
The structure of a dyadic shell model is quite simple, but still captures some interesting features of Euler and Navier-Stokes equations.

Even if the dynamics of a shell model is simpler than the one of Navier-Stokes equation, the existence of solutions for any initial condition can still be a difficult problem, depending on the model considered. If the quantity $\sum_k X_k (t)^2$ (called energy, as it can be identified with the kinetic energy of the fluid) is preserved along the trajectories, for inviscid and unforced shell models with terms interacting only with their nearest neighbors in the system, one has the existence for initial conditions with finite energy, i.e.~for $\ell^2$ initial conditions.
%

With the aim to extend this result  we consider an inviscid and unforced mixed shell model, 
\[
    \frac{d}{dt} X_n (t) = k_n X_{n-1} ^2 (t)- k_{n+1} X_n (t) X_{n+1} (t) - k_n X_{n+1} ^2 (t) + k_{n-1} X_{n-1} (t) X_n (t)  ,
\]
with $k_n = \lambda ^{n}$, $\lambda >1$ for $n >1$ and $k_0 = k_1 = 0$.
This model has both Obukhov-like and Novikov-like non linearities, $X_{n-1} (t) X_n (t)$ and $X_n (t) X_{n+1} (t)$, respectively, and has been discussed also in~\cite{JeoLi2015, Jeo2019}.
In this paper, we are choosing coefficients $k_n$ so that a suitable class of Gaussian measures $\{\mu_r\}_{r>0}$ is invariant for the system.

The goal is to prove that, given $r>0$, there exists a solution for $\mu_r$-a.e. initial conditions, i.e. given $\bar{x} \in \mathbb{R}^{\infty}$, $\mu_r$-almost surely there exists a function $X(t)$ that solves component-wise the system with initial condition $\bar{x}$ in the integral form.  Such a function is obtained as a limit, to do this we consider first a sequence of finite dimensional systems with increasing dimension that approximates the mixed shell model in the Galerkin sense, for any of those finite dimensional systems we have existence and uniqueness of solutions. Hence for any system of the sequence of any dimension $N$ we can introduce a Gaussian measure $\mu_r ^N$ and a random variable $U^N_r$ that picks, randomly with law $\mu_r ^N$, $\bar{x}\in \mathbb{R}^{\infty}$ and gives the unique solution of the system with initial condition $\bar{x}$. The laws of said functions (embedded in the same proper space) have a weak limit, thanks to a combination of Aubin-Lions lemma and Prohorov theorem, and using Skorokhod representation theorem one can pass from a weak limit to an almost sure limit. 
In fact, we take inspiration from the techniques introduced by Albeverio and coauthors in \cite{N1,N2,N3} and later used by Flandoli and Flandoli and Gatarek in \cite{NN1,N6}, adapting them here for a turbulence shell model.

Uniqueness for the mixed shell model is a big issue. One has uniqueness results for shell models in the classical dyadic model with Novikov-like non linearity assuming a positive initial condition. However these do not hold when starting from initial conditions with infinitely-many negative components, a case in which we have infinitely-many solutions (see \cite{1,2, NN3}). 
However these models became a testing bench for techniques of regularisation by noise, since uniqueness can be restored, under suitable conditions, when the model is perturbed with multiplicative noise, see~\cite{BarFlaMorPAMS2010, Bia2013}, or the review paper~\cite{BiaFla2020}.
%
%
Another related topic of research on shell models has to do with blow-up in finite time of solution, in the general framework of the study of global existence of solutions in fluid dynamics. For more along this direction, we refer to~\cite{CheFriPav2007JMP, Che2008TAMS}, and related papers.

Since with our method we work with random initial conditions and solutions having a centred Gaussian measure as law, we are not able to apply the arguments done before for other shell models to extract a unique solution, moreover we don't expect to have the uniqueness since it looks an environment more similar to the one in the classical dyadic model with infinitely-many negative components.

All the arguments of this work can be replicated on a turbulence tree model with a suitable choice of coefficients. The first version of a turbulence tree model, with a nonlinearity of Novikov type, was given in \cite{13}, and was later studied in \cite{NN2} and \cite{NN4}, and we use the same notation of those latter works. A first discussion of the backward cascade as a dyadic tree model can be found in~\cite{BiaFlaMetMor2020}. Usually tree models extend the dynamics of a shell models and they mimic better the phenomenon of bigger eddy splitting in smaller eddies with energy transfer. As done in the dyadic environment, we prove the existence $\mu_r$-almost surely for any initial condition $\bar{x}$ of a component-wise solution of a tree model that has $\mu_r$ as invariant measure for the dynamic.

Let us conclude this introduction by giving a brief overview of the contents of this paper. In Section~\ref{sec:model}, we properly define the model we are dealing with, prove existence of solution and show explicitly a family of invariant measures. Section~\ref{sec:randomic} introduces solutions with random initial conditions, and includes estimates on the norms of such solutions. Section~\ref{sec:compactness} is centered on a compactness result, needed in the following Section~\ref{sec:randsol}, where we put together all the previous results to show the existence of solutions for almost every initial condition. The final Section~\ref{sec:tree} shows how to extend these results for the much more challenging tree model.

\end{section}

\begin{section}{Suitable coefficients for the model}\label{sec:model}

The mixed shell model that we consider in this work is the infinite dimensional dynamical system defined for $t\in [0,T]$ described for each $n\in\mathbb{N}\setminus \{0\}$ by the following equations:
$$
    \frac{d}{dt} X_n (t) = k_n X_{n-1} ^2 (t)- k_{n+1} X_n (t) X_{n+1} (t) - k_n X_{n+1} ^2 (t) + k_{n-1} X_{n-1} (t) X_n (t)  ,
$$
with $k_n = \lambda ^{n}$, $\lambda >1$ for $n >1$ and $k_0 = k_1 = 0$.

The mixed shell model is formally conservative, in the sense that an energy quantity $\mathcal{E}(t) = \sum_n X_n ^2 (t)$ is formally conserved by the equations of the dynamic, as it will be clear later in this section. 
Thanks to this property we get the existence of solutions for $\ell ^2$ initial conditions but we are not able at the moment to extend this results for all initial conditions using only the energy conservation. 
However, the energy conservation is our starting building block to get at least the existence and uniqueness of solutions for the approximating Galerkin sequence of systems, since in each finite system we get trivially that every initial condition belongs to the space $\ell^2$. This is why we will look to a conservative Galerkin approximation.  
So from now on we will call $N$-dimensional shell model the $N$-dimensional dynamical system described by the following equations, for $1 \leq n\leq N$:
$$
  \frac{d}{dt} X_n (t) = k_n X_{n-1} ^2 (t)- k_{n+1} X_n  (t) X_{n+1} (t) - k_n X_{n+1} ^2 (t) + k_{n-1} X_{n-1} (t) X_n (t)  ,
$$
with $k_n = \lambda ^{n}$, $\lambda >1$ for $1<n<  N$ and $k_0 = k_1 = k_N = k_{N+1} = 0$.
 
The last equations are well defined even if apparently there are terms with no meaning: since the coefficient for those terms is 0, actually there is no need to give meaning to non defined terms and the dynamic system has effective dimension of $N$.

\begin{definition}
    With the notation used above, a function $X \in {C} ^1 ([0,T];\mathbb{R} ^N)$ is said to be a solution of the N-dimensional shell model with initial condition $X_0 \in \mathbb{R}^N$ if for each $1\leq n\leq N$ it satisfies the equation  
    $$
        \frac{d}{dt} X_n (t) = k_n X_{n-1} ^2 (t) - k_{n+1} X_n (t) X_{n+1} (t) - k_n X_{n+1} ^2 (t) + k_{n-1} X_{n-1} (t) X_n (t)  ,
    $$
    and if $X(0)=X_0.$
\end{definition}

The Galerkin approximation is built to let every system of the approximating sequence to be conservative, and below we give a proof.

\begin{definition}
    For any solution of the $N$- dimensional shell model we define the kinetic energy $\mathcal{E} (t)$ as the sum
    $$
        \mathcal{E} (t) = \sum_{n=1} ^N X_n ^2 (t)=\|X(t)\|_2^2.
    $$
\end{definition}

\begin{prop}
    The kinetic energy $\mathcal{E} (t)$ is invariant for all the solutions of the N-dimensional shell model.
\end{prop}

\begin{proof}
    We have to show that $\frac{d}{dt} \mathcal{E} (t) = 0,$
    \begin{align*}
        \frac{d}{dt} \mathcal{E} (t) =& \frac{d}{dt} \sum_{n=1} ^N X_n ^2 (t)\\
        =&2 \sum_{n=1}^N (k_n X_n X_{n-1} ^2 - k_{n+1} X_i ^2 X_{n+1} -k_n X_n X_{n+1} ^2 + k_{n-1} X_{n-1} X_n ^2 )\\
        =& 2 (k_1 X_0 ^2 X_1 - k_{N+1} X^2 _N X_{N+1} -k_N X_N X_{N+1} ^2 +k_0 X_0 X_1 ^2 ) =0.   \qedhere
    \end{align*}
\end{proof}

Thanks to the conservation of the kinetic energy we can now show the existence (and uniqueness) of solutions of the $N$-dimensional approximation for any initial condition. 

\begin{prop}
    For any $x_0 \in \mathbb{R}^N$ there exists a unique solution of the N-dimensional shell model with initial condition $X (0) = x_0 $.
\end{prop}

\begin{proof}
    The truncated system satisfies the Cauchy-Lipschitz theorem for the existence of solutions on a ODEs system, 
    and so it admits a local solution $X^N$ on $[0,\delta]$ for some $\delta >0$. We can extend the local solution to a global solution using the energy bound: 
    \begin{equation*}
        \sum_{n=1} ^N X_n ^2 (t) = \sum_{n=1} ^N X_n ^2 (0). \qedhere
    \end{equation*}
\end{proof}

As we mentioned before, the infinite shell model admits a solution for $\ell^2$ initial condition.

\begin{theorem}
    With the same notation of this section, consider the infinite dimensional shell model
    \begin{align*}
        &\frac{d}{dt} X_n (t) = k_n X_{n-1} ^2 (t)- k_{n+1} X_n (t) X_{n+1} (t) - k_n X_{n+1} ^2 (t) + k_{n-1} X_{n-1} (t) X_n (t) \\
        &X(0) = \bar{X}.
    \end{align*}
    For any initial condition $\bar{X} \in \ell^2$ there exists at least a solution $X(t)$ on $[0,T]$.
\end{theorem}
\begin{proof}
    We will use the Ascoli-Arzel\`{a} theorem. For any $\bar{X}$, we consider the sequence of solutions for $N\geq 1$ of the $N$-dimensional shell model $\tilde{X}^N (t)$ obtained considering as initial condition $\bar{X}^N$ the first $N$-entries of $\bar{X} $.
    With abuse of notation we consider all the functions $X^N(t)$ embedded in the same infinite dimensional space by taking the value 0 on the empty entries.
    For every fixed $j$ and $t\in [0,T]$ it holds:
    \begin{itemize}
        \item Uniform boundedness of $\{ X^N _j (t)\}_{N\in \mathbb{N}}$ for both $N$ and $t$:
        $$ 
            |X_j^N (t)|^2 \leq \| X^N (t)\|_2 ^2 =\|X^N(0)\|^2_2 \leq \| \bar{X} \|_2 ^2 .
        $$
        \item Equi-Lipschitzianity of $\{ X^N _j (t)\}_{N\in \mathbb{N}}$ with respect to $N$:
        \begin{align*}
            |\frac{d}{dt} X_j^N (t) |\leq & |k_j X^2_{j-1}(t)|+|k_{j+1} X_j (t) X_{j+1} (t)|+|k_j X^2_{j+1}|+|k_{j-1} X_{j-1} (t) X_{j} (t)| \\
            \leq & \| \bar{X}\|_2 ^2 (2k_j +k_{j+1} +k_{j-1}).
        \end{align*}
    \end{itemize}
    Ascoli-Arzel\`{a} theorem implies for each fixed $j$ the existence of a converging subsequence in $C([0,T])$, i.e. it is possible to find indices $\{ N_k ^j , k\in \mathbb{N}\}$ such that 
    $$
        \sup_{t\in [0,T]} |X_j^{N_k ^j} (t) - X_j (t) |\rightarrow 0 
    $$ 
    for fixed $j$ as $k\rightarrow\infty$.
    
    The sequence $N_{\bullet} ^j$ can be chosen so that $N_{\bullet} ^{j+1}$ is a subsequence of $N_{\bullet} ^j$ itself. By a standard diagonal argument we can extend the convergence to all $j$. 
    If we consider indices $N_k = N_k^k$, we are extracting a common subsequence such that 
    $$
        \sup_{t\in [0,T]} |X_j^{N_k} (t) -X_j(t)|\rightarrow 0
    $$
    for all $j\geq 0$, as $k\rightarrow \infty$.

    Lastly it is straightforward to check that the limit obtained via Ascoli-Arzel\`a theorem, $X(t)$, is a solution for our system, using the equation in the integral form.
\end{proof}

Differently from many other conservative shell models, the mixed model has the key property of the existence of invariant Gaussian measures. At this point we introduce the Gaussian Measure  $\mu^N _r$ on $\mathbb{R}^N$  as 
$$
        \mu^N _r = \bigotimes _{i=1} ^N \mathcal{N} (0,r^2).
    $$

Before proving that the Gaussian measure $\mu^N_r$ is invariant for the trajectories of the $N$-dimensional shell model we need first a technical lemma. So let $\varphi_t (x) $ be the solution with initial conditions $x$.
Let also $P_{t}$ and $P_{t}^{\ast}$ be the semigroups:
\[
    \left(  P_{t}g\right)  \left(  x\right)  :=g\left(  \varphi_{t}\left(x\right)  \right)
\]%
\begin{align*}
    \left(  P_{t}^{\ast}\mu\right)  \left(  g\right)    & := 
    \int_{\mathbb{R}^{N}} g \left( \varphi_{t}\left(  x\right)  \right)  \mu\left(  dx\right)  \\
    & =\mu\left(  P_{t}g\right),
\end{align*}
where $g$ is bounded measurable and $\mu$ is a probability measure. We have
used the notation $\mu\left(  f\right)  :=\int_{\mathbb{R}^{N}}f\left(
x\right)  \mu\left(  dx\right)  $. Notice that $P_{t}^{\ast}\mu$ is a new
probability measure. 

\begin{lemma}
\label{31}
    Let $X,Y$ be two classes of smooth functions $g:\mathbb{R}^{N}\rightarrow
    \mathbb{R}$ with the following properties:%
    \[
        P_{t}\left(  X\right)  \subset Y\text{ for every }t\geq0
    \]%
    \[
        \frac{d}{dt}\left(  P_{t}^{\ast}\mu\right)  \left(  g\right)  |_{t=0}=0\text{
        for all }g\in Y.
    \]
    Then%
    \[
        \frac{d}{dt}\left(  P_{t}^{\ast}\mu\right)  \left(  g\right)  =0\text{ for all
        }g\in X\text{ and }t\geq0\text{.}%
    \]
\end{lemma}

\begin{proof}
    First notice that the following holds%
    \[
        \left(  P_{t+\epsilon}^{\ast}\mu\right)  \left(  g\right)  =\left(
        P_{\epsilon}^{\ast}\mu\right)  \left(  P_{t}g\right)
    \]
    since we have that 
    \begin{align*}
        \left(
        P_{\epsilon}^{\ast}\mu\right)  \left(  P_{t}g\right) &
        = \int_{\mathbb{R}^N} \left( P_tg \right)\left(  \varphi_{\varepsilon} (x) \right) \mu(dx) \\
        & =\int_{\mathbb{R}^N}g(\varphi_t (\varphi_{\varepsilon}(x)))\mu(dx) \\
        &=\int_{\mathbb{R}^N}g(\varphi_{t+\varepsilon} (x))\mu(dx) \\
        &=\left(  P_{t+\epsilon}^{\ast}\mu\right)  \left(  g\right).
    \end{align*}

    Then we get the thesis by the following computation: 
    \begin{align*}
        \frac{d}{dt}\left(  P_{t}^{\ast}\mu\right)  \left(  g\right)    &
        =\lim_{\epsilon\rightarrow0}\frac{\left(  P_{t+\epsilon}^{\ast}\mu\right)
        \left(  g\right)  -\left(  P_{t}^{\ast}\mu\right)  \left(  g\right)}{\epsilon}\\
        & =\lim_{\epsilon\rightarrow0}\frac{\left(  P_{\epsilon}^{\ast}\mu\right)
        \left(  P_{t}g\right)  -\mu\left(  P_{t}g\right)  }{\epsilon}\\
        & =\frac{d}{ds}\left(  P_{s}^{\ast}\mu\right)  \left(  P_{t}g\right)
        |_{s=0}=0. \qedhere
    \end{align*}
\end{proof}

Hence we can go for the main result of this section:

\begin{prop}
\label{32}
    For all $r>0$, the measure $\mu ^N _r$ is invariant for the N-dimensional shell model. 
\end{prop}

\begin{proof}
    Let $b:\mathbb{R}^N \rightarrow \mathbb{R}^N $ be the vector field, $$b_n(x) = k_n x_{n-1} ^2 - k_{n+1} x_n x_{n+1} - k_n x_{n+1} ^2 + k_{n-1} x_{n-1} x_n.$$
    If $\varphi_t (x) $ is the solution with initial conditions $x$, it would be sufficient to have for each $g : \mathbb{R} ^N \rightarrow \mathbb{R}$ regular enough:
    $$
        \frac{d}{dt} \int_{\mathbb{R} ^N } g(\varphi _t (x) ) \mu _r ^N (dx) = 0,
    $$
    to prove that $\mu _r ^N $ in an invariant measure for the system. Note that by Lemma~\ref{31} it is sufficient to prove the equality for $t=0$. We have
    \[
        \frac{d}{dt} \int_{\mathbb{R} ^N } g(\varphi _t (x) ) \mu _r ^N (dx)
        = \int_{\mathbb{R} ^N } \bigtriangledown g(\varphi _t (x) ) \frac{d}{dt} \varphi_t (x) \mu _r ^N (dx) = 0 ,
    \] 
    looking for $t=0$
    $$ 
        \int_{\mathbb{R} ^N } \bigtriangledown g(x) b(x) \mu _r ^N (dx) = 0,
    $$
    let $f(x)$ be such that $\mu_r ^N (dx)= f(x) dx$, so
    $$ 
        \int_{\mathbb{R} ^N } \bigtriangledown g(x) b(x) f(x) dx = 0,
    $$
    and applying Gauss-Green formula we have 
    $$
        \int_{\mathbb{R} ^N } g(x) \divergenza (b(x) f(x)) dx =0 ,
    $$
    that gives 
    $$
        \divergenza (b(x) f(x)) = 0.
    $$ 
    
    For Gaussian measures $\mu_r ^N $we have $f(x) = c e^{\frac{-\| x \| ^2}{r^2}} $ for some constant $c>0$, so we have to show that
    $$
        \divergenza(b(x)f(x)) = \sum_{i=1} ^N \frac{\partial}{\partial x_i } [b_i(x)f(x)]  =
    $$
    $$
        c e^{{\frac{-\| x \|^2}{r^2}}}\sum_{n=1} ^N  [(r^{-2} x_n )(k_n x_{n-1} ^2 - k_{n+1} x_i x_{n+1} -k_i x_{n+1} ^2 +k_{n-1} x_{n-1} x_n) -$$ $$-k_{n+1} x_{n+1} + k_{n-1} x_n] = 0.
    $$
    So it is sufficient to prove the following conditions,
    
    \begin{enumerate}
        \item $\sum_{n=1}^N k_n x_n x_{n-1} ^2 - k_{n+1} x_i ^2 x_{n+1} -k_n x_n x_{n+1} ^2 + k_{n-1} x_{n-1} x_n ^2 =0 ,$
    
        \item $\sum_{n=1} ^N -k_{n+1} x_{i+1} + k_{n-1} x_{n-1} = 0.$
    \end{enumerate}
    The first condition follows directly from $\frac{d}{dt} \mathcal{E} (t) = 0$.

    Computing the second condition we have:
    \begin{multline*}
        \sum_{n=1} ^N -k_{n+1} x_{n+1} + k_{n-1} x_{n-1} = k_0 x_0 + k_1 x_1 -k_N x_N -k_{N+1} x_{N+1} + \\ +\sum_{n=2} ^{N-1} (-k_n + k_n) x_n = 0.
    \end{multline*}
    This concludes the proof.
\end{proof}

The reader can note that one cannot choose more general coefficients of the mixed shell model than we did and have a certain Gaussian measure to be invariant. We mean that if one considers this dynamical system, which is general given the hypothesis of quadratic terms, interactions of term $n$ only with terms $n-1$,$n$, $n+1$ , similar form for any $n$ and energy conservation,
$$
    \frac{d}{dt} X_n (t) = (k_n X_{n-1} ^2 (t)- k_{n+1} X_n (t) X_{n+1} (t)) - (h_n X_{n+1} ^2 (t) + h_{n-1} X_{n-1} (t) X_n (t))  ,
$$
and does the computation to have a Gaussian invariant measure for the Galerkin approximating sequence as done above, necessarily gets $k_n = h_n$:
$$
    \divergenza(b(x)f(x)) = \sum_{i=1} ^n \frac{\partial}{\partial x_i } (b_i(x)f(x) ) =
$$
$$
    c e^{{\frac{-\| x \|^2}{r^2}}}\sum_{i=1} ^n  [(-2 x_i )(k_i x_{i-1} ^2 - k_{i+1} x_i x_{i+1} -h_i x_{i+1} ^2 +h_{i-1} x_{i-1} x_i) -k_{i+1} x_{i+1} + h_{i-1} x_i] = 0.
$$
Since we want the energy conservation along trajectories, we may put 
$$
    h_0 = k_1 = h_{N} = k_{N+1} = 0, 
$$ 
so it holds
$$
    \sum_{i=1}^N (k_i x_i x_{i-1} ^2 - k_{i+1} x_i ^2 x_{i+1} -h_i x_i x_{i+1} ^2 + h_{i-1} x_{i-1} x_i ^2) =0 ,
$$
and we have only to check the following condition
$$
    \sum_{i=1} ^N (-k_{i+1} x_{i+1} + h_{i-1} x_{i-1}) = 0,
$$
hence:
$$
    \sum_{i=1} ^N (-k_{i+1} x_{i+1} + h_{i-1} x_{i-1} ) = h_0 x_0 + h_1 x_1 -k_N x_N -k_{N+1} x_{N+1} + \sum_{i=2} ^{N-1} (-k_i + h_i) x_i.
$$
This leads to $k_i = h_i$ for every $i$ and $k_0 = k_1 =k_N = k_{N+1}=0$. Note that in this way the terms $x_0$ and $x_{N+1}$ disappear from the equations.

\end{section}

\begin{section}{Random initial conditions}\label{sec:randomic}

The first aim of this section is to determine a random environment where the definition of random solution of a deterministic equation makes sense. Then, once we have built the environment we need, we get fundamental estimates on the norm of the random solution for the next section. 
Hence we introduce the following notation:

    Let $(\Omega , \mathcal{F} , P)$ be an abstract probability space, for every $N$ and for 
    $r>0$ let ${Y}^N _r$ be a random variable 
    $$
        {Y}^N _r:(\Omega , \mathcal{F} , P) \rightarrow (\mathbb{R} ^N , \mathcal{B} (\mathbb{R ^N})),
    $$
    with law $\mu_r ^N$.

\begin{definition}
    A set $(\Omega, \mathcal{F},P, U^N_r)$ is said to be a $N$-finite random solution with initial condition 
    ${Y}^N _r$ if $U^N_r$ is defined on the abstract probability space $(\Omega,\mathcal{F},P)\times [0,T]$ 
    to $\mathbb{R}^{\infty}$, all $k$-coordinates of $U^N_r$ are almost surely for each time $t\in[0,T]$ equal 
    to $0$ if $k>N$ and for $k\leq N$ almost surely 
    $$
        {U^N_r} _{(k)} (\omega, t) = {F^N_r} _{(k)} (\omega,t),
    $$
    where, for $\omega\in\Omega$, the function $$F^N _r (\omega) : [0, T] \rightarrow \mathbb {R} ^N$$ is the 
    unique solution of the N-dimensional shell model with initial conditions 
    $$ 
        X(0) = Y^N_r (\omega).
    $$

\end{definition}
Remark that $F^N_r $ is still a random variable from the abstract space $(\Omega, \mathcal{F} , P)$.
 
\begin{prop}
\label{33}
    Let $(\Omega, \mathcal{F},P,U^N_r)$ be a $N$-finite random solution.
    The law of $U^N_r (t) $ is, for any $t\in[0,T]$, 
    $$
        \tilde{\mu}_r^N = \mu_r^N\otimes\bigotimes_{N+1}^{\infty}\delta_0.
    $$
\end{prop}
 
\begin{proof}
    It follows directly from the definition of $U_r^N (t)$ and the invariance of $\mu_r^N$ along the trajectories of the $N$-dimensional shell model.
\end{proof}

At this point it should be clear to the reader that the choice of the abstract space $(\Omega , \mathcal{F} , P)$ is totally arbitrary, it matters only the law of the random variables defined on said space. This opens the door to a future use of the Skorokhod representation theorem to get an almost sure limit in place of a weak limit.

Our next issue to solve is to give a suitable norm to the space of random solutions, in order to be able to speak about convergence in that topology. A natural norm to work with could be the $H^s$ norm defined below, since, roughly speaking, we expect any limit of the finite random solution sequence at time $0$ to almost surely have said norm finite for $s<0$, as we will prove in the next section.
 
\begin{definition}
    We define $(H^s, \| \cdot \|_{H^s})$ as the Hilbert space of sequences $x \in \mathbb{R} ^{\infty} $ satisfying 
    $$ 
        \| x\| _{H^s} = \sqrt{ \sum_n k_n ^{2s} x_n ^2 }< \infty .
    $$
\end{definition}

The following proposition on the $H^s$ norms will be significant in the next section to pass from a compactness result in $L^p (0,T; H^s)$ for any $p>1$ to a compactness result in $C(0,T;H^s)$.
 
\begin{prop}
\label{34}
    Let $s_1 < s < s_0 < 0$. Then there exists $\theta \in (0,1)$ such that 
    $$
        \| x \| _{H^s} \leq  \| x \|_{H^{s_0}} ^{\theta} \| x\|_{H^{s_1}} ^{1-\theta},
    $$
    for every $x \in H^{s_0} $.
\end{prop}

\begin{proof}
    For all $a\in (2s,0)$, $b\in (0,2)$, it holds
    $$
        \| x\|_{H^s} ^2= \sum_n k_n ^{2s} x_n ^2 = \sum_n k_n ^a |x_n| ^b k_n^{2s -a} |x_n| ^{2-b} \leq
    $$
    we use H\"{o}lder inequality for generic $p,q>1$ such that $\frac{1}{p} + \frac{1}{q} =1$.
    $$ 
        \leq \left(\sum_n (k_n ^a |x_n| ^b )^p \right)^{\frac{1}{p}} \left(\sum_n (k_n ^{2s-a} |x_n| ^{2-b} )^q\right)^{\frac{1}{q}}.
    $$
    We can now let $b=\frac{2}{p}$ in order to have $bp =2$ and $(2-b)q = 2$, that gives $p = \frac{2}{b}$ and $q = \frac{2}{2-b}$, since $\frac{2}{p}\in(0,2)$ and since it holds 
    $$
        \left(\frac{2}{b} \right)^{-1} + \left(\frac{2}{2-b} \right)^{-1} = \frac{b}{2} +\frac{2-b}{2} = 1.
    $$
    Hence we get
    $$
        \| x\|_{H^s} ^2 \leq \left(\sum_n k_n ^{ap} x_n ^2\right) ^{\frac{1}{p}}\left(\sum_n k_n ^{(2s-a)q}x_n^2 \right)^{\frac{1}{q}} .
    $$

    Looking at the statement we want to prove, it remains to show that we can choose $a\in(2s,0)$, $p>1$, $q>1$ such that $ap = 2 s_0$, $(2s -a)q = 2 s_1$ and $\frac{1}{p} + \frac{1}{q}=1$. So solving
    $$
        \begin{cases}
            ap=2s_0 \\
            (2s-a)q=2s_1 \\
            \frac{1}{p}+\frac{1}{q}=1 
        \end{cases}
    $$
    we get 
    $$
        p=\frac{s_0-s_1}{s-s_1},
    $$
    hence $p>1$, hence $q>1$. The last check we have to do is for $a$ and we get $a=\frac{2s_0}{p} \in (2s,0)$ since $s_0 <s$ and $p>1$. So we have proved the statement for 
    \[
        \theta =\frac{1}{p}=\frac{s-s_1 }{s_0 - s_1}.  \qedhere
    \]
\end{proof}

Here, with the specific non-linearity of the system, is where the choice of the measure plays a central role. The following estimates hold for the Gaussian measure $\mu_r ^N$ and a quadratic non-linearity, so a reader that is not strictly focused on shell models can skip this part. Outside of the following two estimates the work is pretty independent of the type of non-linearity of the system, so to generalize the method to something else one must first check the existence of estimates similar to the following ones.

The estimates \ref{36} and \ref{37} guarantee the existence of a family of compact set in the topology of $L^p (0,T; H^s)$ suitable to apply Prohorov theorem, as we will see later in Section 4.
Before stating these estimates, we prove the following lemma, that we will need to establish the propositions.

\begin{lemma}
\label{35}
    Let $k_n$ be the coefficients of the shell model, $l\in \mathbb{N} = \{ 0,1,\dots \}$, $q \in \mathbb{R}$ such that $q+l <0$. Let $\varphi $ be a $\mathcal{C}^{\infty}$ function such that $\varphi (x)\neq 0$ for any $x$. Then
    $$ 
        \sum_n \log \varphi (t k_n ^{2q + 2l})
    $$
    is $h$-times differentiable in $t=0$ for any $h\in\mathbb{N}$ and the derivative operation commutes with the sum, we mean that for any $h$ $\frac{d^h}{dt^h} \sum_n \log \varphi (t k_n ^{2q + 2l}) = \sum_n \frac{d^h}{dt^h} \log \varphi (t k_n ^{2q + 2l})$.  
\end{lemma}
\begin{proof}
    We divide this proof in four steps. 
    \begin{description}

        \item[Step 1] 
            Let $\zeta(t) = \varphi(t)^{h_0} \varphi ' (t) ^{h_1} \ldots \varphi^{(k)} (t)^{h_k}$ be a monomial in the variables $\varphi(t), \varphi '(t), \ldots , \varphi ^{(k)} (t)$ of degree $z =\sum_{i=0} ^{k} h_i$. 
            Then its derivative $\frac{d}{dt} \zeta(t)$ is an homogeneous polynomial of degree $z$ in the variables $\varphi(t), \varphi '(t), \ldots , \varphi ^{(k)} (t), \varphi ^{(k+1)} (t)$.
     
            This follows by a straightforward computation:
            $$
                \frac{d}{dt}( \varphi(t)^{h_0} \varphi ' (t) ^{h_1} \ldots \varphi^{(k)} (t)^{h_k} )
                = \sum_i \sum_{|h_i|} h_i \varphi ^{(i+1)} (t) \varphi ^{(i)}(t)^{h_i -1} \prod_{j\neq i} \varphi^{(j)} (t)^{h_j}.
            $$
     
        \item[Step 2]
            We want to prove by induction the following:
            $$
                \frac{d^{k}}{dt^k} \log \varphi (t) = \frac{P_k (\varphi(t), \varphi '(t), \ldots , \varphi ^{(k)} (t))}{\varphi(t)^{2^k}}
            $$
            where $P_k $ is an homogeneous polynomial of degree $2^k$ .
            
            For $k=1$ we have $\frac{d}{dt}\log \varphi (t) = \frac{\varphi ' (t)}{\varphi (t)}$. Assuming the thesis true for $k$ we make the computation
            \begin{multline*}
                \frac{d^{k+1}}{dt^{k+1}} \log \varphi (t)= \frac{d}{dt}\left(\frac{P_k (\varphi(t), \varphi '(t), \ldots , \varphi ^{(k)} (t))}{\varphi(t)^{2^k}} \right) \\
                = \frac{(\frac{d}{dt}P_k (\varphi(t), \varphi '(t), \ldots , \varphi ^{(k)} (t))) \varphi (t)^{2^k} - 
                P_k (\varphi(t), \varphi '(t), \ldots , \varphi ^{(k)} (t)) (2^k -1) \varphi(t)^{2^k -1} \varphi ' (t) }{\varphi (t)^{2^{k+1}}}
            \end{multline*}
            and using Step 1 we have the thesis.
     
        \item[Step 3]
            For a generic function $\xi $ it holds 
            $$
                \frac{d^k}{dt^k} \xi (\lambda t) = \lambda ^k \xi^{(k)} (\lambda t).
            $$
            Hence, looking for $t=0$, it holds 
            $$ 
                \frac{d^k}{dt^k} \xi (\lambda t)_{|t=0} = \lambda ^k \frac{d^k}{ds^k}\xi (s)_{|s=0}.
            $$
     
        \item[Step 4]
            To conclude the lemma we have to show that for every $h\in \mathbb{N}$,
            $$
                \sum_n  \left(\frac{d^h}{dt^h} \log \varphi (t k_n ^{2q + 2l})\right)\Big|_{t=0}<\infty .
            $$
            Indeed 
            $$ 
                \sum_n  \left(\frac{d^h}{dt^h} \log \varphi (t k_n ^{2q + 2l})\right)\Big|_{t=0} = \sum_n (k_n ^{2q + 2l})^h \frac{P_h (\varphi (s) , \ldots ,\varphi^{(h)} (s))}{\varphi(s)^{2^h}}\Big|_{s=0}
            $$
            and since $\frac{P_h (\varphi (s) , \ldots ,\varphi^{(h)} (s))}{\varphi(s)^{2^h}}\Big|_{s=0}$ is a constant, so not depending on $n$, we have the proof.\qedhere
    \end{description}
\end{proof}

\begin{prop}
\label{36}
    Let $(\Omega, \mathcal{F},P, U^N_r)$ be a $N$-finite random solution.
    For every $s<0$, $r>0$, $p>1$, $\epsilon>0$ there exists a constant $C_{\epsilon} >0$, not depending on $N$, such that 
    $$
        P(\| U^N _r  \|_{L^p(0,T;H^s)} \leq C_{\epsilon}) > 1- \epsilon,
        $$
    for each $N\in \mathbb{N}$.
\end{prop}  
\begin{proof}
    We want to prove that for any $p>1$ and for any $\varepsilon >0$ exists $R\in\mathbb{R^+}$ such that
    $$
        P(\| U_r ^N \|  _{L^p (0,T ; H^s )} > R )<\varepsilon.
    $$
    Hence
    $$ 
        P(\| U_r ^N \|  _{L^p (0,T ; H^s )} > R ) =  P(\| U_r^N \| ^p _{L^p (0,T ; H^s )} > R^p ) \leq 
    $$
    now we apply Markov inequality
    \begin{align*} 
        \leq & \frac{1}{R^p } E[\| U_r^N \| ^p _{L^p (0,T ; H^s )}] \\
        = &\frac{1}{R^p } \| U_r^N \| ^p _{L^p (\Omega \times [0,T] , H^s )} = \frac{1}{R^p} \int^T_0 E[\| U_r^N (t,\omega)\| ^p _{H^s}] ds = 
    \end{align*}
    here, thanks to Proposition \ref{33}, we use the time invariance for the law of $U^N _r $
    $$ 
      =\frac{T}{R^p} E [\| U_r^N  (0, \omega) \| ^p _{H^s}] .
    $$
    Hence it is sufficient to show that for any $p>1$ there exists $C\in\mathbb{R^+}$ such that $$ E [\| U_r^N  (0, \omega) \| ^p _{H^s}]<C,$$
    for any $N$, since the proof will follow letting $R\rightarrow\infty$. Note that for each $N$ holds 
    $$ 
      E[\| U_r ^N (0, \omega )\| ^p _{H^s}]\leq E \left[\left| \left(\sum_{n=1} ^{\infty} k_n ^{2s} r^2 W_n ( \omega) \right)^{\frac{p}{2}} \right| \right] 
    $$
    where $W_i \sim \chi^2(1)$, with $\{W_i\}_i$ iid.
    
    So it is sufficient to prove that the random variable 
    $$ 
        Z =\sum_{n\geq 1} k_n ^{2s} r^2 W_n ,
    $$
    has a moment generating function derivable infinite times in 0, this would imply that it has finite $p$-moment for any $p \geq 1$ and this is would give the uniform bound in $N$ for $L^p (0,T;H^s)$ norm we need.
    
    The moment generating function of $Z$, $\psi(t)$ is 
    $$
        \psi(t) = E[e^{t \sum k_n ^{2s} r^2W_n (\omega)}].
    $$ 
    Notice that 
    $$
        \log E [ e^{t\sum_{n=1} ^m k_n ^{2s} r^2 W_n ( \omega)}  ] = \sum_{n=1}^m \log \varphi (t k_n ^{2s} r^2) 
    $$
    for every $m \in \mathbb{N}$, where $\varphi(t k_n ^{2s} r^2)$ is the moment generating function of $k_n ^{2s} r^2 W_n $. If we define the random variables $Z_m = e^{t\sum_{n=1} ^m k_n ^{2s}r^2 W_n ( \omega)}$ we have that for $t\geq 0$ $Z_m$ is an increasing sequence of random variable, and for $t<0$ it is dominated by $1$.
    So for all $t$ we can have $E[ \lim_{m \rightarrow \infty} Z_m] = \lim_{m\rightarrow \infty} E[Z_m]$, hence
    $$
        \log \psi(t) = \sum_n \log \varphi (t k_n^{2s} r^2).
    $$
    
    We are left to show that $\sum_n \log \varphi (t k_n^{2s} r^2)$ is differentiable infinite times in $t=0$, and this is true for Lemma~\ref{35}.
\end{proof}

\begin{prop}
\label{37}
    Let $(\Omega, \mathcal{F},P, U^N_r)$ be a $N$-finite random solution. 
    For every $s<-1$, $r>0$, $p>1$ $\epsilon>0$ there exists a constant $C_{\epsilon} >0$ such that 
    $$
        P(\| U^N _r  \|_{W^{1,p}(0,T;H^s)} \leq C_{\epsilon}) > 1- \epsilon,
    $$
    for each $N\in \mathbb{N}$.
\end{prop}  
\begin{proof}
    Again we have that 
    \begin{align*}
        P(\| U^N_r \| ^p _{W^{1,p} (0,T ; H^{s} )} > R^p ) & \leq \frac{1}{R^p } E[\| U_r^N  \| ^p _{W^{1,p} (0,T ; H^{s} )}] \\ 
        & = \frac{1}{R^p} E[\| \frac{d}{dt} U_r^N \|^p _{L^p (0,T;H^{s})}] .
    \end{align*}
        
    Let $\{W_{1,k},W_{2,k},W_{3,k},W_{4,k},W_{5,k},W_{6,k}\}_{k\in\mathbb{N}}$ be a set of iid Gaussian random variables, with $W_{i,j} \sim \mathcal{N}(0,r^2)$. For each $N$ it holds 
    \begin{multline*}
        E[\| \frac{d}{dt} U_r ^N (0, \omega )\| ^p _{H^s}]\\
         \leq E[(\sum_{n\geq 1} (k_n {W_{1,n}}^2 -k_{n+1} W_{2,n} W_{3,n} - k_n {W_{4,n}} ^2 + k_{n-1} W_{5,n} W_{6,n})^2 k_n ^{2s}    )^{\frac{p}{2}} ] . 
    \end{multline*}

    Moreover, there exists a constant $D>0$ such that 
    \begin{multline*}
        E[(\sum_{n\geq 1} (k_n {W_{1,n}}^2 -k_{n+1} W_{2,n} W_{3,n} - k_n {W_{4,n}} ^2 + k_{n-1} W_{5,n} W_{6,n})^2 k_n ^{2s}    )^{\frac{p}{2}} ] \\
        \leq E[ (D\sum_{n\geq 1} k_n^{2+2s}W_{1,n} ^4)^{\frac{p}{2}} ].
    \end{multline*}
    
    So it is sufficient to prove that the random variable 
    $$ 
        Z = \sum_{n\geq 1} k_n^{2+2s}W_{1,n} ^4  ,
    $$
    has a moment generating function differentiable infinite times in 0, this would imply that it has finite $p$-moment for any $p \geq 1$.

    Using the same argument of Proposition \ref{36} we have that, if $\psi(t)$ is the moment generating function of $Z$, 
    $$
        \log \psi(t) = \sum_n \log \varphi (t k_n^{2+2{s}} r ), 
    $$
    where $\varphi (t k_n^{2+2{s}} r )$ is the moment generating function of $k_n ^{2+2s} {W_{1,n}} ^4$. Moreover, for Lemma \ref{35}, $\sum_n \log \varphi (t k_n^{2+2{s}} r )$ is derivable infinite times in $t=0$. 
\end{proof}

\end{section}

\begin{section}{A compactness result}\label{sec:compactness}

Considering a $N$-finite random solution $(\Omega, \mathcal{F},P, U^N_r)$, we need a compactness criterium for the family of laws  $\{ \mathcal{L} (U_r^N)\}_{N \in \mathbb{N}}$ to extract a converging subsequence $\lim_{K \rightarrow \infty} U_r^{N_k} = U_r ^{\infty}$ in law, 
and then without loss of generality we would have a limit almost surely $\lim_{K \rightarrow \infty} U_r^{N_k} = U_r^{\infty}$ up to changing the abstract space $(\Omega, \mathcal F , P)$ via Skorokhod Theorem.

As anticipated in the previous section, it is natural to extract the limit in the topology of $L^P (0,T;H^s)$ for $s < 0$, since if $s<0$, let $U_r ^{\infty}$ be any limit of $U_r ^N$, we would have $U_r ^{\infty} (0) \in  H^s$ $P$ - almost surely, since for Markov inequality and monotone convergence $$P(\| U_r^{\infty} (0) \|_{H^s} > C) \leq \frac{1}{C}E[\| U_r ^{\infty} (0)\|_{H^s}] = \frac{1}{C} \sum  r^2 \lambda^{(n-1)2s}$$ and the series on the right converges since it is geometric with $\lambda > 1$ and $s < 0$.
 
To prove the convergence of a subsequence we want to use the Prohorov compactness theorem. Thanks to the estimates done on the previous section, we satisfy the condition of the classical Aubin-Lions theorem, that guarantees the existence of proper compact sets, useful for a further application of Prohorov theorem.

Hence we give to the reader the Aubin-Lions theorem in the form we will use.

\begin{theorem}[Aubin-Lions]\label{thm:aubin_lions}
    Let $B_0 \subset  B \subset B_1$ be Banach spaces, $B_0$ and $B_1$ reflexive, with compact embedding of $B_0$ in $B$ and a continuous embedding of $B$ into $B_1$. Let $p \in (1, \infty )$. Let $X$ be the space
    \[
        X = L^p (0,T ; B_0 ) \cap W^{1 , p} (0, T ; B_1 ),
    \]
    endowed with the natural norm. Then the embedding of $X$ in $L^p (0,T ; B)$ is compact.
\end{theorem}

More words can be spent about this theorem, the form of the Aubin-Lions theorem we use is the classical one, but it would work also for some weakened hypothesis: one can be sharper and ask for the space to have a $W^{\alpha,p}$ regularity, with $\alpha \in (0,1)$ instead of $\alpha =1$.
More precisely, we mention this adaptation of Flandoli-Gatarek \cite{N6}.
  
\begin{theorem}
    Let $B_0 \subset B \subset B_1$ be Banach spaces, with $B_0$, $B_1$ reflexive, a compact embedding of $B_0$ into $B$ and a continuous embedding of $B$ into $B_1$. Let $p\in(1,\infty)$ and $\alpha \in (0,1)$ be given. Let $X$ be the space
    $$
        X= L^p (0,T;B_0)\cap W^{\alpha,p} (0,T; B_1)
    $$
    endowed with the natural norm. Then the embedding of $X$ in $L^p (0,T;B)$ is compact.
\end{theorem}
  
For a proof of this we refer to Theorem 2.1 of \cite{N6}.

In the previous section we have proved the estimates for all $p\in(1,\infty)$. This has done for a future use of a combination of Proposition \ref{34} together with the following result from Simon, in order to extend later a relative compactness in $L^p(0,T;H^s)$ for all $p\in(1,\infty)$ to a relative compactness in $C(0,T; H^s)$.
  
\begin{prop}
\label{39}
    Suppose we have $X \subset B \subset Y$ Banach spaces, with a compact embedding $X \rightarrow Y$. Suppose also there exists a $\theta \in (0,1)$ and a $M$ such that
    $$
        \| v\|_B \leq M \| v\|_X ^{1 -\theta} \| v\|_Y ^{\theta}, 
    $$
    for any $v \in X \cap Y$. Let $F$ be bounded in $L^{p_0} (0,T; X)$ and $\frac{\partial F}{\partial t}$ be bounded in $L^{r_1} (0, T ; Y)$, with $1 \leq p_0 \leq \infty$, $1 \leq r_1 \leq \infty$.
    
    If $\theta (1 - \frac{1}{r_1 }) > \frac{1 -\theta}{p_0 }$ then $F$ is relatively compact in $C(0,T ;B)$.
\end{prop}

For a proof of this result we refer to Corollary 8, section 10 of \cite{N5}. 
  
We get finally all the tools and estimates to conclude the section with our main compactness result about the sequence of laws of the random solutions.
 
\begin{theorem}
\label{tightness}
    For fixed $r>0$, let $(\Omega, \mathcal{F},P, U^N_r)$ be a $N$-finite random solution. The family of law $\{ \mathcal{L} (U^N _r ) \}_{N\in \mathbb{N}}$ is tight in $L^P (0, T, H^s )$ for every $p>1$, $s<0$. Moreover, $\{ \mathcal{L} (U^N _r ) \}_{N\in \mathbb{N}}$ is tight in $C(0,T; H^s)$.
\end{theorem}
\begin{proof}
    We use Aubin-Lions Theorem~\ref{thm:aubin_lions} with $B_0 = H^s$, $s<0$, $B_1 = H^{s_1 }$, $s_1 < -1$, $B = H^{s^ *}$, $s_1 < s^* < s$ and $p=r$.  
    So the set $$K_{R_1 , R_2} = \{ X | \| X \|_{L^p (0,T ; H^s)} \leq R_1 ,  \| X \|_{W^{1,p} (0,T; H^{s_1})} \leq R_2 \}$$ is relatively compact in $L^p (0,T; H^{s^*})$.
    
    From Propositions \ref{36} and \ref{37}, we have that for every $\varepsilon >0$ there exists a constant $C_{\varepsilon}$ such that :
    \begin{itemize}
        \item $P (\| U_r ^N \|_{L^p (0,T ; H^s)} \leq C_{\varepsilon}) \geq 1- \varepsilon  $ for all $N \in \mathbb{N}$,
        \item $P(\| U_r ^N \|_{W^{1,p} (0,T; H^{s_1})} \leq C_{\varepsilon} ) \geq 1- \varepsilon  $ for all $N \in \mathbb{N}$.
    \end{itemize}

    So, given $\varepsilon >0$ there exist $R_1 (\varepsilon)$, $R_2 (\varepsilon)$ such that the family of laws satisfies 
    $$
        \{ \mathcal{L} (U_r ^N)\} \subset \{ \mu \in Pr (L^p (0,T; H^s )) | \mu ({K^c _{R_1, R_2}}) \leq \varepsilon      \},
    $$
    hence, since $\bar{K} _{R_1, R_2} \supseteq K _{R_1, R_2}$ we have $\bar{K} ^c _{R_1, R_2} \subseteq K ^c _{R_1, R_2}$ and then 
    $$
    \{ \mathcal{L} (U_r^N)\} \subset \{ \mu \in Pr (L^p (0,T; H^s )) |  \mu ({\bar{K}^c _{R_1, R_2}}) \leq \varepsilon      \},
    $$
    so the family of laws $\{ \mathcal{L} (U_r^N)\}_{N \in \mathbb{N}}$ is tight in the topology of $L^p (0, T; H^s)$ for any $p \geq 1$.
   
    For the tightness in $C(0,T ; H^s)$ we can use Proposition \ref{39} with $X= H^{s_0}$, $B = H^s$, $Y = H^{s_1}$ and $\theta$ from Proposition \ref{34}. 
    If we let $p_0 , r_1 \rightarrow \infty$ the condition of Proposition \ref{39} 
    $$
        \theta (1 - \frac{1}{r_1 }) > \frac{1 -\theta}{p_0 }
    $$ 
    is trivial, so we have that, with the same arguments of the $L^p$ case, the family of laws $\{ \mathcal{L} (U_r^N)\}_{N \in \mathbb{N}}$ is tight in the topology of $C (0, T; H^s)$ .
\end{proof}
 
\begin{cor}
\label{compactness}
    For $s<0$, there exists a subsequence $n_k \subset \mathbb{N}$ such that the laws of $\{U^{n_k}_r\} $ converge with the topology of $L^p (0,T; H^s)$ for every $p\geq 1$ and with the topology of $C(0,T; H^s )$.
\end{cor}
\begin{proof}
    We have only to apply Prohorov Theorem to the results of Theorem \ref{tightness}.
\end{proof}

\end{section}

\begin{section}{Existence of random solution}\label{sec:randsol}

In this section we put all together to have the existence of solutions for almost every initial condition, where almost every is respect to the probability measure $\mu_r$ on the infinite dimensional space of initial conditions. So first we have to establish what is for us a ``random solution'' for the infinite dimensional model, and then prove that the limit extracted in the previous section, after having changed it from a limit in law to a limit almost surely thanks to Skorokhod theorem, effectively fits the definition of random solution.
 
\begin{definition}
    A $(\Omega,\mathcal{F},P,X)$ is said to be a random solution for the infinite shell model if $(\Omega,\mathcal{F},P)$ is an abstract probability space, $X  : (\Omega, \mathcal{F}, P) \times [0, T] \rightarrow \mathbb{R}^{\infty}$ and for almost every $\omega \in \Omega$, $X(\omega, t)$ satisfies for every $i\in \mathbb{N}$ and $t\in[0,T]$
    $$ 
        X_i (\omega, t) = X_i (\omega, 0) +\int_0 ^t k_i X_{i-1} ^2 (\omega, s)ds - \int_0 ^t k_{i+1} X_i (\omega, s) X_{i+1}(\omega, s) ds $$  $$-\int_0 ^t k_i X_{i+1} ^2 (\omega, s)ds +\int_0 ^t k_{i-1}X_{i-1}(\omega, s) X_i (\omega, s) ds.
    $$
\end{definition}

The following theorem represents the goal of the work of the chapter.

\begin{theorem}
    For fixed $r>0$, consider a sequence $\{(\Omega,\mathcal{F},P,U^N_r)\}_{N\geq1}$ of $N$-finite random solutions, up to replace the abstract space $(\Omega, \mathcal{F} , P)$ with another probability space $(\Omega ', \mathcal{F} ', P ')$, there exists a subsequence $n_k \in {\mathbb{N}}$ such that $P'$-almost surely $U_{r} ^{nk}$ converges to a function $U^{\infty} _r$ in $L^p (0,T; H^s)$ for every $p>1$ and in $C(0,T; H^s)$, for every $s<0$. Moreover, $(\Omega ',\mathcal{F} ', P',U^{\infty} _r)$ is a random solution for the infinite shell model.
\end{theorem}

\begin{proof}
    From Corollary \ref{compactness} we have the existence of a subsequence $\{ n_k\}_{k\in\mathbb{N}} \subset \mathbb{N}$ such that the sequence of laws of the random variables $U_r ^{n_k}$ converges in the topology of $C(0,T;H^s)$ and $L^p (0,T; H^s)$ for every $p\geq 1$.
  
    Since we have the convergence for any $s<0$, we can consider the space $$H^{0-} = \bigcap_{s<0} H^s,$$ endowed with the metric generated by the distance 
    $$
        d(x,\tilde{x}) = \sum_{n=1}^{\infty}2^{-n} (\|x-\tilde{x} \|_{H^{-\frac{1}{n}}} \wedge 1).
    $$ 
    Note that with this metric we have that $x_n \rightarrow ^dx \Leftrightarrow x_n \rightarrow^{H^s} x$ for any $s<0$.
    
    We can assume, using Skorokhod representation Theorem, that almost surely $\tilde{U}_r ^{n_k}$ converges to $U_r ^{\infty}$, up to replace the abstract space $(\Omega, \mathcal{F} , P )$ where $\tilde{U}_r ^N$ are defined with another abstract probability space $(\Omega ', \mathcal{F} ' , P ' )$, in the topology of $C(0,T;H^{0-})$ and $L^p (0,T; H^{0-})$ for every $p\geq 1$.
    The new sequence of random variables $\tilde{U}_r ^{n_k}$ has the same law of ${U}_r ^{n_k}$, this means that for every $\varphi$ measurable function it holds $E[\varphi(\tilde{U}_r ^{n_k})]=E[\varphi({U}_r ^{n_k})]$.
    
    Consequently, considering the operator 
    $$
        F_i(x(\omega,t))=x_i (\omega,t) -x_i (\omega,0) -\int_0 ^t k_i x^2 _{i-1} (\omega, s)ds + \int_0^t k_{i+1} x_i (\omega,s) x_{i+1} (\omega,s)ds$$  $$+\int_0^t k_i x^2_{i+1} (\omega,s)ds - \int_0^t k_{i-1} x_{i-1}(\omega,s)x_i (\omega,s)ds,
    $$
    we have that $E[|F_i (\tilde{U}_r ^{n_k})|]=E[|F_i ({U}_r ^{n_k})|]$, hence $\tilde{U}_r ^{n_k}$ is still almost surely a solution of the finite dimensional shell model. So for now on we consider without loss of generality ${U}^n_r =\tilde{U}^n_r$.
 
    Now let $\{Y^N\}$ and $Y^{\infty}$ such that for each $\varepsilon >0$ there exists a $N_0 $ such that for every $N>N_0$ we have 
    $$
        \sup_{[0,T]} \sum_{n \geq 1} k_n^{2s} (Y^N_n - Y^{\infty} _n) ^2 < \varepsilon,
    $$ 
    this implies that 
    $$
        \sup_{[0,T]} | Y_n^N - Y_n ^{\infty}| < (\frac{\varepsilon}{k_n^{2s}})^{\frac{1}{2}} \dot{=} \varepsilon ' _n .
    $$
    Hence we have:
    \begin{enumerate}
        \item 
            \begin{align*}
                \int_0^t |{Y_n^{N } (s)}^2 - {Y_n^{\infty } (s)}^2  |ds \leq & \int_0^t 2|Y_n ^N (s)| | Y_n^N (s) - Y^{\infty}_n (s)| + |Y^N_n (s) - Y^{\infty} _n (s)|^2 ds\\
                \leq & \int_0^t 2 \sup_{s\in [0,T]}|Y_n ^N (s)| \sup_{s\in[0,T]}| Y_n^N (s) - Y^{\infty}_n (s)| \\
                & + \sup_{s \in [0,T]}|Y^N_n (s) - Y^{\infty} _n (s)|^2 ds \\
                \leq & t(2 \sup_{s\in [0,T]}|Y_n ^N (s)| \varepsilon _n ' + {\varepsilon_n'} ^2).
            \end{align*}
            
            \clearpage
             
        \item   
            \begin{align*}
                \int_0 ^t |Y^N _n (s) Y^N_{n-1} (s) - Y^{\infty} _n (s) Y^{\infty} _{n-1} (s)| ds \leq & \int_0 ^t |Y^N_n (s) - Y^{\infty} _n(s)||Y^N_{n-1} (s) - Y^{\infty}_{n-1} (s)|\\
                &+|Y^N_{n-1} (s) - Y^{\infty} _{n-1} (s)||Y^N_n (s)|\\
                & +  |Y^N_n (s) - Y^{\infty} _n (s)||Y^N_{n-1} (s)| ds\\
                \leq & \int_0 ^t \sup_{s\in[0,T]}|Y^N_n (s) - Y^{\infty} _n(s)| \sup_{s\in[0,T]}|Y^N_{n-1} (s) - Y^{\infty}_{n-1} (s)|\\
                & + \sup_{s\in[0,T]}|Y^N_{n-1} (s) - Y^{\infty} _{n-1} (s)|\sup_{s\in[0,T]}|Y^N_n (s)| \\
                & + \sup_{s\in[0,T]}|Y^N_n (s) - Y^{\infty} _n (s)|\sup_{s\in[0,T]}|Y^N_{n-1} (s)| ds\\
                \leq & t (\varepsilon '_n \varepsilon ' _{n-1} + \sup_{s\in[0,T]} |Y_n ^N (s)| \varepsilon ' _{n-1} + \sup_{s\in[0,T]}|Y^N _{n-1} (s)| \varepsilon ' _n).
            \end{align*}
    \end{enumerate}
  
    Since almost surely $ \sup_{s\in[0,T]}U_n ^N (\omega, s) < \infty$ for every $n,N$, we have that almost surely 
    \begin{multline*}
        \int_0 ^t k_i ({U^{n_k}_{i-1}} ^2 (\omega, s)   -{U^{\infty}_{i-1}} ^2 (\omega, s) )ds -\int_0 ^t k_i ( {U^{n_k}_{i+1}} ^2 (\omega, s) - {U^{\infty}_{i+1}} ^2 (\omega, s) )ds \\ - \int_0 ^t k_{i+1} (U^{n_k}_i (\omega, s) U^{n_k}_{i+1}(\omega, s) - U^{\infty}_i (\omega, s) U^{\infty}_{i+1}(\omega, s))ds\\
         +\int_0 ^t k_{i-1} (U^{n_k}_{i-1}(\omega, s) U^{n_k}_i (\omega, s) - U^{\infty}_{i-1}(\omega, s) U^{\infty}_i (\omega, s) )ds
    \end{multline*}
    goes to $0$ as $k \rightarrow \infty$, so $U^{\infty} _r $ is a random solution for the infinite shell model.
\end{proof}
 
\end{section}

\begin{section}{Invariant measure method on the tree}\label{sec:tree}

In this section we show how to replicate the previous work on a tree model. The tree-like structure have eddies as nodes, and a node is son of another node if the corresponding
eddy is formed by a split of the corresponding eddy of the father.

We denote by $J$ the set of nodes, and if $j \in J$  we call $\mathcal{O} _j$ the set of offspring of $j$. We made assumption that every eddies has the same biggest eddy as ancestor, so we can classify eddies in ``generations'' or ``levels''.
Level 0 is made by only the biggest eddy $\emptyset \in J$, level 1 is made by the eddies produced by the one in level 0 and so on. We will denote the generation of an eddy $j$ by $|j |$. The father of an eddy $j$ will be denoted as $\bar\jmath$. To construct the dynamic we associate at each eddy $j$ a non-negative intensity $X_j (t)$.
The following system is the one we want to explore in this section:
$$
    \frac{d}{dt} X_j = \alpha (c_j X^2 _{\bar\jmath} - \sum_{k\in \mathcal{O}_j} c_k X_j X_k) - \beta (d_{\bar\jmath} X_{\bar\jmath} X_j - \sum_{k\in \mathcal{O}_j} d_j X_k ^2),
$$
where $\alpha,\beta,c_i , d_i \in \mathbb{R}$, $\alpha c_j - \beta  d_j =0$ for any $|j|\geq 1$,  $c_0 = d_{\bar{0}} = d_0 = 0$. 

As usual we consider the truncated version of the infinite dimensional system. So, for $N\in \mathbb{N}$ we define:
$$
    \frac{d}{dt} X_j = \alpha (c_j X^2 _{\bar\jmath} - \sum_{k\in \mathcal{O}_j} c_k X_j X_k) - \beta (d_{\bar\jmath} X_{\bar\jmath} X_j - \sum_{k\in \mathcal{O}_j} d_j X_k ^2),
$$
where $\alpha,\beta,c_i , d_i \in \mathbb{R}$, $\alpha c_j - \beta  d_j =0$ for any $1\leq |j|\leq N-1$,  $c_0 = d_{\bar{0}} = d_0 = 0$ and if $|j| \geq N$, $c_j = d_j = 0$.

For both system (infinite and truncated ones) we ask that there exists an $M\in\mathbb{N}$ such that $\sum_{k\in\mathcal{O}_j} 1 \leq M$ for every $j\in\mathbb{N}$ and there exists a $\lambda>1$ such that
$$
    c_j =\lambda c_{ \bar\jmath} 
$$
for every $j\geq2$ for the infinite system, for every $2\leq j \leq Q$ for the truncated one and $c_j=\lambda$ for $|j|=1$.

For now on we put $Q=\sum_{|k|\leq N} 1$.

\begin{theorem}
    The $Q$-dimensional dynamic system defined above is conservative, in the sense that the kinetic energy $\sum_{|k|\leq N} X_k ^2$ is preserved.
\end{theorem}
\begin{proof}
    We compute the derivative of the energy:
    $$
        \frac{d}{dt} \left(\sum_{|j| \leq N} X_j ^2\right) = 2 \sum _j \alpha (c_j X_j X_{\bar\jmath}^2 -\sum_{k\in \mathcal{O}_j} c_k X_j ^2 X_k) - \beta (d_{\bar\jmath} X_{\bar\jmath} X_j ^2 -\sum_{k \in \mathcal{O}_j} d_j X_j X_k ^2) =
    $$
    the two sums are telescoping since every term compares once as father and once as son, so
    $$
        = 2(\alpha c_0 X_0 X_{\bar{0}} ^2  -  \sum_{|j|=N, k\in \mathcal{O}_j } \alpha c_k X^2 _j X_k  -\beta d_0 X_0 ^2 X_{\bar{0}} + \sum_{|j|=N, k\in \mathcal{O}_j} \beta d_j X_j X^2 _k) = 0,
    $$
    thus concluding the proof.
\end{proof}

\begin{prop}
    For $N\in \mathbb{N}$ the system:
    $$
        \frac{d}{dt} X_j = \alpha (c_j X^2 _{\bar\jmath} - \sum_{k\in \mathcal{O}_j} c_k X_j X_k) - \beta (d_{\bar\jmath} X_{\bar\jmath} X_j - \sum_{k\in \mathcal{O}_j} d_j X_k ^2),
    $$
    $$ 
        X(0)=X_0
    $$
    where $\alpha,\beta,c_i , d_i \in \mathbb{R}$, $\alpha c_j - \beta  d_j =0$ for any $1\leq |j|\leq N-1$,  $c_0 = d_{\bar{0}} = d_0 = 0$ and if $|j| \geq N$, $c_j = d_j = 0$ admits a solution for any initial condition $X_0$.
\end{prop}
\begin{proof}
    Since we have the energy conservation, the proof is the same of the one done for the dyadic case.
\end{proof}

So, with same argument of the dyadic case, the infinite dimensional tree model defined in this section admits a solution for any $\ell^2$ initial condition. On the tree model considered in this section we want to improve the existence result from an $\ell^2$ initial condition result to an almost every initial condition, with respect to a Gaussian measure to infinite dimensional space of initial conditions.

\begin{theorem}
    The product Gaussian measure $$\mu _r ^Q = \bigotimes_{|j|\leq N} \mathcal{N} (0, r^2)$$ is invariant for the system.
\end{theorem}
\begin{proof}
    Consider the dynamical system in the following form:
    $$
        \dot{X} = b(x).
    $$
    As proved for the diadic model it suffices to have
    $$
        \divergenza (b(x) f(x)) = 0.
    $$ 
    For Gaussian measures $\mu_r ^Q $ we have $f(x) = c e^{\frac{-\| x \| ^2}{r^2}} $ for some constant $c>0$, so 
    $$
        \divergenza (b(x)f(x)) = \sum_{i=1} ^Q \frac{\partial}{\partial X_i } b(x)f(x) _i =
    $$
    $$
        c e^{-\frac{\| X\|^2}{2}} [ \sum _j \alpha (c_j X_j X_{\bar\jmath}^2 -\sum_{k\in \mathcal{O}_j} c_k X_j ^2 X_k) - \beta (d_{\bar\jmath} X_{\bar\jmath} X_j ^2 -\sum_{k \in \mathcal{O}_j} d_j X_j X_k ^2 )]+
    $$
    $$
        + c e^{-\frac{\| X\|^2}{2}} [\sum_{|j|\leq N} -\alpha(\sum_{k\in \mathcal{O}_j} c_k X_k) -\beta d_{\bar\jmath } X_{\bar\jmath}],
    $$
    where the first term of the sum is equal to $0$ as proved in the previous theorem, the second term is equal to 
    $$ 
        c e^{-\frac{\| X\|^2}{2}} ( -\beta d_{\bar{0}} X_{\bar{0}} -\beta d_0 X_0 - \alpha \sum_{|j|=N , c\in \mathcal{O}_j} c_k X_k - \sum_{1\leq |j|\leq N-1} \alpha c_j + \beta d_j) =0.
    $$
\end{proof}

\begin{definition}
    Let $(\Omega , \mathcal{F} , P)$ be an abstract probability space, for every $N$ and for $r>0$ let ${Y}^Q _r$ be a random variable 
    $$
        {Y}^Q _r:(\Omega , \mathcal{F} , P) \rightarrow (\mathbb{R} ^Q , \mathcal{B} (\mathbb{R }^Q)),
    $$
    with law $\mu_r ^Q$.
\end{definition}

\begin{definition}
    A set $(\Omega, \mathcal{F},P, U^Q_r)$ is said to be a $Q$-finite random solution if $U^Q_r$ is defined on the abstract probability space $(\Omega,\mathcal{F},P)\times [0,T]$ to $\mathbb{R}^{\infty}$, all $k$-coordinates of $U^Q_r$ are almost surely for each time $t\in[0,T]$ equal to $0$ if $k>Q$ and for $k\leq Q$ almost surely 
    $$
        {U^Q_r} _{(k)} (\omega, t) = {F^Q_r} _{(k)} (\omega,t),
    $$
    where, for $\omega\in\Omega$, the function 
    $$
        F^Q _r (\omega) : [0, T] \rightarrow \mathbb {R} ^Q
    $$ 
    is the unique solution of the $Q$-dimensional truncated tree model with initial conditions 
    $$ 
        X(0) = Y^Q_r (\omega).
    $$ 
    Remark that $F^Q_r $ is still a random variable from the abstract space $(\Omega, \mathcal{F} , P)$.
\end{definition}

\begin{prop}
\label{433}
    Let $(\Omega, \mathcal{F},P,U^Q_r)$ be a $Q$-finite random solution.
    The law of $U^Q_r (t) $ is, for any $t\in[0,T]$, 
    $$
        \tilde{\mu}_r^Q = \mu_r^Q\otimes\bigotimes_{Q+1}^{\infty}\delta_0.
    $$
\end{prop} 
\begin{proof}
    It follows directly from the definition of $U_r^Q (t)$ and the invariance of $\mu_r^Q$ along the trajectories of the $Q$-dimensional truncated tree model.
 \end{proof}

We equip the space with a similar norm to the one of the previous section.
 
\begin{definition}
    We define $(H^s, \| \cdot \|_{H^s})$ as the Hilbert space of sequences $x \in \mathbb{R} ^{\infty} $ satisfying 
    $$ 
        \| x\| _{H^s} = \sqrt{ \sum_n c_n ^{2s} x_n ^2 }< \infty .
    $$
\end{definition}

\begin{lemma}
\label{435}
    Let $c_n$ be the coefficients of the tree model, $k\in \mathbb{N} = \{ 0,1,\dots \}$, $q \in \mathbb{R}$ such that $q+k <0$. Let $\varphi $ be a $\mathcal{C}^{\infty}$ function. Then
    $$ 
        \sum_n \log \varphi (t c_n ^{2q + 2k})
    $$
    is $h$-times differentiable in $t=0$ for any $h\in\mathbb{N}$ and the derivative operation commutes with the sum, we mean that for any $h$ $\frac{d^h}{dt^h} \sum_n \log \varphi (t c_n ^{2q + 2k}) = \sum_n \frac{d^h}{dt^h} \log \varphi (t c_n ^{2q + 2k})$.  
\end{lemma} 
\begin{proof}
    First note that 
    $$
        \frac{d^h}{dt^h} \log \varphi(t) = \frac{c \varphi^{(h)} (t) \varphi^{2k-1} (t) + r(t)}{\varphi(t)^{2^{k-1}}},
    $$
    where $r(t)$ is a polynomial where monomial are product of $\varphi^{(j)} (t)^h$ and $\varphi(t)^l$ with $j<h$ and $h+l \leq 2^{h-1}$.
    
    Then note that
    $$ 
        \frac{d^h}{dt^h} \xi(\lambda t) = \lambda^h \xi^{(h)} (\lambda t),
    $$
    so 
    $$ 
        \frac{d^h}{dt^h} \xi(\lambda t)_{|t=0} = \lambda^h \frac{d^h}{ds^h} \xi ( s) _{|s=0}.
    $$

    The lemma is true if for every $h$ we show:
    $$
        \sum_n |\frac{d^h}{dt^h}\log \varphi (t c_n ^{2q + 2k})|_{|t=0} < \infty.
    $$

    Combining what we said above, we have 
    $$
        \sum_n |\frac{d^h}{dt^h}\log \varphi (t c_n ^{2q + 2k})|_{|t=0} = \sum_n (c_n^{2q +2k})^h \frac{c \varphi^{(h)} (s) \varphi^{2h-1} (s) + r(s)}{\varphi(s)^{2^{h-1}}} _{|s=0} ,
    $$
    that converges for $k+q <0$. 
\end{proof}

The following two proposition are the ones imposing such restrictions on the coefficients $c_k$. Until now one can have put $\sum_{i\in\mathcal{O}_j} c_i =\lambda c_j$, having a result on a more general tree. 
Sadly the two following estimates don't work with non-increasing $c_k$ coefficients and we need all $c_k$ increase by a geometric factor from a generation to another.
 
\begin{prop}
\label{436}
    Let $(\Omega, \mathcal{F},P, U^N_r)$ be a $Q$-finite random solution.
    For every $s<0$, $r>0$, $p>1$, $\epsilon>0$ there exists a constant $C_{\epsilon} >0$, not depending on $N$, such that 
    $$
        P(\| U^Q _r  \|_{L^p(0,T;H^s)} \leq C_{\epsilon}) > 1- \epsilon,
    $$
    for each $Q=\sum_{|i|\leq N}1$ with $N\in\mathbb{N}\setminus \{0\}$.
\end{prop}
\begin{proof}
    We want to prove that for any $p>1$ and for any $\varepsilon >0$ exists $R\in\mathbb{R^+}$ such that
    $$
        P(\| U_r ^Q \|  _{L^p (0,T ; H^s )} > R )<\varepsilon.
    $$
    Hence
    $$ 
        P(\| U_r ^Q \|  _{L^p (0,T ; H^s )} > R ) =  P(\| U_r^Q \| ^p _{L^p (0,T ; H^s )} > R^p ) \leq 
    $$ 
    now we apply Markov inequality
    $$ 
        \leq \frac{1}{R^p } E[\| U_r^Q \| ^p _{L^p (0,T ; H^s )}]  = 
    $$ 
    $$ 
        = \frac{1}{R^p } \| U_r^Q \| ^p _{L^p (\Omega \times [0,T] , H^s )} = \frac{1}{R^p} \int^T_0 E[\| U_r^Q (t,\omega)\| ^p _{H^s}] ds = 
    $$
    here, thanks to Proposition \ref{433}, we use the time invariance for the law of $U^Q _r $
    $$ 
        =\frac{T}{R^p} E [\| U_r^Q  (0, \omega) \| ^p _{H^s}] .
    $$
    
    Hence it is sufficient to show that for any $p>1$ exists $C\in\mathbb{R^+}$ such that 
    $$ 
        E [\| U_r^Q  (0, \omega) \| ^p _{H^s}]<C,
    $$
    for any $Q$, since the proof will follow letting $R\rightarrow\infty$. Note that for each $Q$ holds 
    $$ 
        E[\| U_r ^Q (0, \omega )\| ^p _{H^s}]\leq E [| (\sum_{n=1} ^{\infty} c_n ^{2s} r^2 W_n ( \omega) )^{\frac{p}{2}} | ]  
    $$
    where $W_i \sim \chi^2(1)$, with $\{W_i\}_i$ iid.
    
    So it is sufficient to prove that the random variable 
    $$ 
        Z =\sum_{n\geq 1} c_n ^{2s} r^2 W_n ,
    $$
    has a moment generating function derivable infinite times in 0, this would imply that it has finite $p$-moment for any $p \geq 1$ and this is would give the uniform bound in $Q$ for $L^p (0,T;H^s)$ norm we need.

    The moment generating function of $Z$, $\psi(t)$ is 
    $$
        \psi(t) = E[e^{t \sum c_n ^{2s} r^2W_n (\omega)}].
    $$ 

    Note that 
    $$
        \log E [ e^{t\sum_{n=1} ^m c_n ^{2s} r^2 W_n ( \omega)}  ] = \sum_{n=1}^m \log \varphi (t c_n ^{2s} r^2) 
    $$
    for every $m \in \mathbb{N}$, where $\varphi(t c_n ^{2s} r^2)$ is the moment generating function of $c_n ^{2s} r^2 W_n $. If we define the random variables $Z_m = e^{t\sum_{n=1} ^m c_n ^{2s}r^2 W_n ( \omega)}$ we have that for $t\geq 0$ $Z_m$ is an increasing sequence of random variables, and for $t<0$ it is dominated by $1$.
    So for all $t$ we can have $E[ \lim_{m \rightarrow \infty} Z_m] = \lim_{m\rightarrow \infty} E[Z_m]$, hence
    $$
        \log \psi(t) = \sum_n \log \varphi (t c_n^{2s} r^2).
    $$

    There is still to show that $\sum_n \log \varphi (t c_n^{2s} r^2)$ is differentiable infinite times in $t=0$, and this is true for Lemma \ref{435}.
\end{proof}

\begin{prop}
\label{437}
    Let $(\Omega, \mathcal{F},P, U^Q_r)$ be a $Q$-finite random solution.
    For every $s<-1$, $r>0$, $p>1$ $\epsilon>0$ there exists a constant $C_{\epsilon} >0$ such that 
    $$
        P(\| U^Q _r  \|_{W^{1,p}(0,T;H^s)} \leq C_{\epsilon}) > 1- \epsilon,
    $$
    for each $Q=\sum_{|i|\leq N}1$ with $N\in\mathbb{N}\setminus \{0\}$.
\end{prop}

\begin{proof}
    Again we have that 
    $$ 
        P(\| U^Q_r \| ^p _{W^{1,p} (0,T ; H^{s} )} > R^p ) \leq \frac{1}{R^p } E[\| U_r^Q  \| ^p _{W^{1,p} (0,T ; H^{s} )}]  = 
    $$ 
    $$ 
        = \frac{1}{R^p} E[\| \frac{d}{dt} U_r^Q \|^p _{L^p (0,T;H^{s})}] .
    $$
    
    Let $\{W_{1,k},\ldots, W_{2M+2,k}\}_{k\in\mathbb{N}}$ be a set of iid Gaussian random variables, with $W_{i,j} \sim \mathcal{N}(0,r^2)$. For each $N$ holds 
    $$ 
        E[\| \frac{d}{dt} U_r ^Q (0, \omega )\| ^p _{H^s}]\leq
    $$
    $$
        \leq E[(\sum_{n\geq 1} (c_n ({W_{1,n}}^2 +\ldots+{W_{M+2,n}}^2))^2 c_n ^{2s}    )^{\frac{p}{2}} ] . 
    $$
    
    Moreover, there exists a constant $D>0$ such that 
    $$ 
        E[(\sum_{n\geq 1} (c_n ({W_{1,n}}^2 +\ldots+W_{M+2,n}))^2 c_n ^{2s}    )^{\frac{p}{2}} ] \leq
    $$
    $$
        \leq E[ (D\sum_{n\geq 1} c_n^{2+2s}W_{1,n} ^4)^{\frac{p}{2}} ].
    $$
    
    So it is sufficient to prove that the random variable 
    $$ 
        Z = \sum_{n\geq 1} c_n^{2+2s}W_{1,n} ^4  ,
    $$
    has a moment generating function differentiable infinite times in 0, this would imply that it has finite $p$-moment for any $p \geq 1$.

    Using the same argument of Proposition \ref{436} we have that, if $\psi(t)$ is the moment generating function of $Z$, 
    $$
        \log \psi(t) = \sum_n \log \varphi (t c_n^{2+2{s}} r ), 
    $$
    where $\varphi (t c_n^{2+2{s}} r )$ is the moment generating function of $c_n ^{2+2s} {W_{1,n}} ^4$. Moreover, for Lemma \ref{435}, $\sum_n \log \varphi (t c_n^{2+2{s}} r )$ is derivable infinite times in $t=0$. 
\end{proof}

From now on the guideline is the same of the last two sections, having estimates from Propositions \ref{436} and \ref{437} we can extract tightness for the laws of the $Q$-finite random solution and then with Prohorov theorem we can have a converging subsequence of said laws, both in $L^p(0,T;H^s)$ and $C(0,T;H^s)$ topology.
 
\begin{theorem}
\label{tightness2}
    For fixed $r>0$, let $(\Omega, \mathcal{F},P, U^Q_r)$ be a $Q$-finite random solution. The family of law $\{ \mathcal{L} (U^Q _r ) \}$ is tight in $L^P (0, T, H^s )$ for every $p>1$, $s<0$. Moreover, $\{ \mathcal{L} (U^Q _r ) \}$ is tight in $C(0,T; H^s)$.
\end{theorem}
\begin{proof}
    
    Having the estimates given by Propositions \ref{436} and \ref{437}, the proof is straightforward using the same arguments of the dyadic case.
\end{proof}

\begin{cor}
\label{compactness2}
    For $s<0$, there exists a subsequence $n_k \subset \mathbb{N}$ such that the laws of $\{U^{n_k}_r\} $ converge with the topology of $L^p (0,T; H^s)$ for every $p\geq 1$ and with the topology of $C(0,T; H^s )$.
\end{cor} 
\begin{proof}
    We have only to apply Prohorov Theorem to the results of Theorem \ref{tightness2}.
\end{proof}

\begin{definition}
    A set $(\Omega,\mathcal{F},P,X)$ is said to be a random solution for the tree model if $(\Omega,\mathcal{F},P)$ is an abstract probability space, $X  : (\Omega, \mathcal{F}, P) \times [0, T] \rightarrow \mathbb{R}^{\infty}$ and for almost every $\omega \in \Omega$, $X(\omega, t)$ satisfies for every $i\in \mathbb{N}$ and $t\in[0,T]$
    $$ 
        X_j (\omega, t) = X_j (\omega,0) + \alpha \int_0^t  (c_j X^2 _{\bar\jmath} (\omega,s) - \sum_{k\in \mathcal{O}_j} c_k X_j (\omega,s) X_k(\omega,s))ds - 
    $$ 
    $$
        \beta\int_0^t (d_{\bar\jmath} X_{\bar\jmath} (\omega,s) X_j (\omega,s) - \sum_{k\in \mathcal{O}_j} d_j X_k ^2(\omega,s)) ds.
    $$
\end{definition}

\begin{theorem}
    For fixed $r>0$, consider a sequence $\{(\Omega,\mathcal{F},P,U^Q_r)\}$ of $Q$-finite random solutions, up to replace the abstract space $(\Omega, \mathcal{F} , P)$ with another probability space $(\Omega ', \mathcal{F} ', P ')$, there exists a subsequence $n_k \in {\mathbb{N}}$ such that $P'$-almost surely $U_{r} ^{nk}$ converges to a function $U^{\infty} _r$ in $L^p (0,T; H^s)$ for every $p>1$ and in $C(0,T; H^s)$, for every $s<0$. Moreover, $(\Omega ',\mathcal{F} ', P',U^{\infty} _r)$ is a random solution for the tree model.
\end{theorem}
\begin{proof}
    From Corollary \ref{compactness2} we have the existence of a subsequence $\{ n_k\}_{k\in\mathbb{N}} \subset \mathbb{N}$ such that the sequence of laws of the random variables $U_r ^{n_k}$ converges in the topology of $C(0,T;H^s)$ and $L^p (0,T; H^s)$ for every $p\geq 1$. 
    
    Since we have the convergence for any $s<0$, we can consider the space $$H^{0-} = \bigcap_{s<0} H^s,$$ endowed with the metric generated by the distance 
    $$
        d(x,\tilde{x}) = \sum_{n=1}^{\infty}2^{-n} (\|x-\tilde{x} \|_{H^{-\frac{1}{n}}} \wedge 1).
    $$ 
    Note that with this metric we have that $x_n \rightarrow ^dx \Leftrightarrow x_n \rightarrow^{H^s} x$ for any $s<0$.
    
    We can assume, using Skorokhod representation Theorem, that almost surely $\tilde{U}_r ^{n_k}$ converges to $U_r ^{\infty}$, up to replace the abstract space $(\Omega, \mathcal{F} , P )$ where $\tilde{U}_r ^N$ are defined with another abstract probability space $(\Omega ', \mathcal{F} ' , P ' )$, in the topology of $C(0,T;H^{0-})$ and $L^p (0,T; H^{0-})$ for every $p\geq 1$. 
    
    The new sequence of random variables $\tilde{U}_r ^{n_k}$ has the same law of ${U}_r ^{n_k}$, this means that for every $\varphi$ measurable function it holds $E[\varphi(\tilde{U}_r ^{n_k})]=E[\varphi({U}_r ^{n_k})]$. Hence, considering the operator 
    $$
        F_i(x(\omega,t))=x_i (\omega,t) -x_i (\omega,0) - \alpha \int_0^t  (c_j x^2 _{\bar\jmath} (\omega,s) - \sum_{k\in \mathcal{O}_j} c_k x_j (\omega,s) x_k(\omega,s))ds + 
    $$ 
    $$
        + \beta\int_0^t (d_{\bar\jmath} x_{\bar\jmath} (\omega,s) x_j (\omega,s) - \sum_{k\in \mathcal{O}_j} d_j x_k ^2(\omega,s)) ds
    $$
    we have that $E[|F_i (\tilde{U}_r ^{n_k})|]=E[|F_i ({U}_r ^{n_k})|]$, hence $\tilde{U}_r ^{n_k}$ is still almost surely a solution of the truncated tree model. So for now on we consider without loss of generality ${U}^n_r =\tilde{U}^n_r$.
  
    It is then straightforward to check that the limit solves the equation in the integral form, as done in the dyadic case, so $U^{\infty} _r $ is a random solution for the tree model.
 \end{proof}

\end{section}

\bibliographystyle{abbrv}\bibliography{monta}
\end{document}